\documentclass[12pt,A4,reqno]{amsart}
\usepackage{amsfonts}
\usepackage{mathrsfs}
\usepackage[T1]{fontenc}
\usepackage{amssymb}
\usepackage{amsmath,amscd}
\usepackage{color}
\usepackage{epsf}
\usepackage{graphicx}
\usepackage[curve]{xypic}
\usepackage{tikz}
\usepackage{dsfont}
\usetikzlibrary{matrix}

\theoremstyle{plain}
\newtheorem{thm}{Theorem}[section]
\newtheorem{lem}[thm]{Lemma}
\newtheorem{prop}[thm]{Proposition}
\newtheorem{cor}[thm]{Corollary}

\theoremstyle{definition}

\newtheorem{rem}[thm]{Remark}
\newtheorem{exam}{Example}

\newcommand{\Q}{\mathbb Q}
\newcommand{\R}{\mathbb R}
\newcommand{\Z}{\mathbb Z}
\newcommand{\F}{\mathbb F}
\newcommand{\nn}{\vskip 0.2cm}
\newcommand{\n}{\vskip 0.1cm}

\newcommand*{\longhookrightarrow}{\ensuremath{\lhook\joinrel\relbar\joinrel\rightarrow}}

\begin{document}

\title [\ ] {On Homology Roses and the $\mathrm{D(2)}$-problem}

\author{Xifeng Jin$^\dagger$}
\address{$^\dagger$Department of Mathematics, Nanjing University, Nanjing, 210093, P.R.China}
\email{xifengj@gmail.com}

\author{Yang Su$^\ddag$}
\address{$^\ddag$Institute of Mathematics, Hua Loo-Keng Key Laboratory of Mathematics, 
 Chinese Academic of Science, Beijing, 100190, 
     P.R.China}
\email{suyang@math.ac.cn}

\author{Li Yu*}
\address{*Department of Mathematics and IMS, Nanjing University, Nanjing, 210093, P.R.China}
 \email{yuli@nju.edu.cn}
 \thanks{
 *The author is partially supported by
 Natural Science Foundation of China (grant no.11001120). This work is also
 funded by the PAPD (priority academic program development) of Jiangsu higher education institutions.}


\subjclass[2010]{57Q12, 57M07, 57M10, 57M20}

\keywords{Homology rose, deficiency, group cohomology, Carlsson conjecture, 
$\mathrm{D(2)}$-problem, gap, efficiency, Cockcroft property}

\begin{abstract}
 For a commutative ring $R$ with a unit, an $R$-homology rose is a topological space whose homology groups with $R$-coefficients agree with those of a
 bouquet of cirlces. In this paper, we study some special properties of covering spaces and 
 fundamental groups of $R$-homology roses, from which we obtain 
  some result supporting the Carlsson conjecture on free $(\Z_p)^r$ actions.
 In addition, for a group $G$ and a field $\F$, we
  define an integer called the $\F$-gap of $G$, which is an obstruction
     for $G$ to be realized as the fundamental group of a $2$-dimensional
     $\F$-homology rose. Furthermore,
      we discuss how to search candidates of the counterexamples of
     Wall's $\mathrm{D(2)}$-problem among $\F$-homology roses and $\F$-acyclic spaces.
 \end{abstract}

\maketitle

 \section{Introduction}
    Let $R$ be a commutative ring with a unit. An \emph{$R$-homology rose $B$ with $m$-petals} 
    ($m\geq 1$) is
     a topological space with
     $H_*(B;R) \cong H_*(\bigvee_m S^1;R)$ where $\bigvee_m S^1$ is a 
     bouquet of $m$ circles. 
     If $m=1$, $B$ is called
     an \emph{$R$-homology circle}.
     Our main concern is when $R$ is $\Z$ or a field $\F$. 
     When $R=\F_p = \Z\slash p\Z$ ($p$ is a prime),
      $B$ is called a \emph{mod-$p$ homology rose} (\emph{circle}).
     In addition,  an \emph{$R$-acyclic space} is a topological 
     space $B$ with 
     $\widetilde{H}_*(B;R) = 0$.
 \nn

    Many important spaces
      that occur in geometry and topology are $R$-homology roses. 
    For example, for any knot $N$ in a $3$-sphere $S^3$, its complement
    $S^3 \backslash N$ is a $\Z$-homology circle. \nn   
   
    \begin{thm} \label{Thm:Covering}
      Suppose $(\Z_p)^r$ acts cellularly and freely on a finite CW-complex $X$.
     Then the orbit space $X\slash (\Z_p)^r$ is a
      mod-$p$ homology rose if and only if $X$ is a mod-$p$ homology
      rose. 
    \end{thm}
        
    \begin{rem}
     Theorem~\ref{Thm:Covering} actually follows from~\cite[Proposition 1.2]{BrownKahn77}.
  But the proof of Theorem~\ref{Thm:Brown} given in~\cite{BrownKahn77} uses some results from 
   Tate cohomology theory, which may not be so familiar to some people.  
   So in the following,
    we give a proof of Theorem~\ref{Thm:Covering} via
     some elementary tools of algebraic topology.    
    \end{rem}

  A group $G$ acting cellularly on a CW-complex $X$ means that $G$
    acts by homeomorphisms of $X$
   that map each cell homeomorphically to another cell. If moreover
   $G$ acts freely, then the orbit space $X\slash G$ is also a
   CW-complex and so
   we can think of $X$ as a regular covering space over $X\slash G$
   with deck transformation group $G$. So we have 
   another way to state Theorem~\ref{Thm:Covering}.
   \nn

  \noindent \textbf{Theorem~\ref{Thm:Covering}'.} A regular $(\Z_p)^r$-covering $X$ of a finite CW-complex $K$ is a mod-$p$ homology
  rose if and only if $K$ is a mod-$p$ homology rose.\nn
  
     \begin{rem}
     If we do not require the CW-complex $X$ in Theorem~\ref{Thm:Covering} to be 
     finite, 
     the conclusion of Theorem~\ref{Thm:Covering}
      may not be true. For example, $\Z_2$ can act freely 
     $X=S^{\infty} \times S^1$ with orbit space $\R P^{\infty} \times S^1$.
    \end{rem}
    \n
       For any topological space $X$ and any filed $\F$, let 
        $$b_i(X;\F) :=\dim_{\F} H_i(X;\F).$$
    Similarly, for any group $\Gamma$, let     
      $$ b_i(\Gamma;\F) = \dim_{\F} H_i(\Gamma;\F), \ \forall\, i\geq 0,$$
    where $H_i(\Gamma;\F)=H_i(K(\Gamma, 1);\F)$ is the group homology of $\Gamma$ with $\F$-coefficients.
   \n

    For an $\F$-homology rose $K$ with $m$ petals, it is easy to see that
     \begin{equation} \label{Equ:b_1}
       b_1(\pi_1(K);\F) =m,\  \ b_2(\pi_1(K);\F) = 0.
     \end{equation}

 \n

   For a regular covering space $\xi: X \rightarrow B$ with deck transformation
   group $G$, $\xi_*: \pi_1(X) \rightarrow \pi_1(B)$ maps $\pi_1(X)$ injectively into
   $\pi_1(B)$. And if we identify $\pi_1(X)$ with its image under $\xi_*$,
     we have $\pi_1(B)\slash \pi_1(X) \cong G$.
     If we assume that $B$ is a finite CW-complex 
     and a mod-$p$ homology rose, Theorem~\ref{Thm:Covering} tells us that
     any regular $(\Z_p)^r$-covering $X$ over $B$ is also a mod-$p$ homology rose. 
     So we have the following corollary.
\n

    \begin{cor} \label{Cor:Subnormal-Series}
       If a finite CW-complex $K$ is a mod-$p$ homology rose with $m$ petals, then
       there exists an
  infinitely long subnormal series of $\pi_1(K)$
   \begin{equation} \label{Equ:Subnormal-Series}
      \pi_1(K)=\Gamma_0 \rhd \Gamma_1 \rhd \cdots \rhd \Gamma_k \rhd \cdots 
    \end{equation}  
    where $\Gamma_i\slash \Gamma_{i+1} \cong \Z_p$ and $b_1(\Gamma_i;\F_p)=p^i(m-1)+1$, 
    $b_2(\Gamma_i;\F_p)=0$ for all $i\geq 0$.     
    \end{cor}

    In addition, there is an easy
     algebraic criterion to judge whether a group can be realized as the 
      fundamental group of an $R$-homology rose.\nn
     
  \begin{prop} \label{Prop:Fund-Group}
    For any commutative ring $R$ with a unit, a group $\Gamma$ can be realized as
      the fundamental group of an $R$-homology rose with $m$ petals
      if and only if $H_1(\Gamma;R) = R^m$ and $H_2(\Gamma;R)=0$.
  \end{prop}
  
   Indeed, we can show that for a finitely presentable group $\Gamma$ with 
   $H_1(\Gamma;R) = R^m$ and $H_2(\Gamma;R)=0$,
   there always exists a finite 
   $3$-complex $B$ with fundamental group $\Gamma$ and $B$ is an $R$-homology rose.
   But generally speaking, we can not choose $B$ to be $2$-dimensional. 
   Actually, for any field $\F$,
   there is an obstruction (called the \emph{$\F$-gap} of $\Gamma$)
   to the existence of a $2$-dimensional $\F$-homology 
   rose $B$ with $\pi_1(B)\cong \Gamma$. We will study the properties of
   the $\F$-gap of a group and then use this notion to
   study Wall's $\mathrm{D(2)}$-problem among $\F$-homology roses and $\F$-acyclic spaces. 
   \n
   From Theorem~\ref{Thm:Covering} and Proposition~\ref{Prop:Fund-Group}, we obtain the following
   theorem.\n
   
   \begin{thm} \label{thm:Carlsson}
  Let $B$ be a connected finite CW-complex with $H_2(\pi_1(B);\F_p)=0$.
     Then for any regular $(\Z_p)^r$-covering $X$ over $B$, we have
     $b_1(X;\F_p) \geq 2^r-1$.
   \end{thm}

  \begin{cor} \label{Cor:Carlsson}
    Let $B$ be a connected finite CW-complex with $H_2(B;\F_p)=0$. Then
     for any regular $(\Z_p)^r$-covering $X$ over $B$, we have
     $\sum^{\infty}_{i=0} b_i(X;\F_p) \geq 2^r$.
  \end{cor}
  
   The above result gives supporting evidences for the 
   Carlsson conjecture which is an original motivation for our paper. Recall
   that the \emph{Carlsson conjecture} claims that 
   if $(\Z_p)^r$ can act freely 
    on a finite CW-complex $X$, then  
       $\sum^{\infty}_{i=0} b_i(X;\F_p) \geq 2^r$.
   In particular, if the free $(\Z_p)^r$-action on $X$ is cellular, 
   the orbit space $X\slash (\Z_p)^r$
    is also a finite CW-complex and $X$ is a regular $(\Z_p)^r$-covering 
    over $X\slash (\Z_p)^r$. 
   The Carlsson conjecture is also called the \emph{toral rank conjecture} in some literature and 
   remains open so far. The 
   reader is referred to~\cite{Adem04, AllPupp93, Hanke09} for more information on the 
 Carlsson conjecture.
       \n
       
  In addition, from Proposition~\ref{Prop:Fund-Group} 
  we can classify all the finitely generated abelian groups that occur as the fundamental
  groups of mod-$p$ homology roses or mod-$p$ acyclic spaces 
  (Proposition~\ref{prop:Homology-Rose-Abelian} and Proposition~\ref{Prop:Acyclic-1}). 
  Moreover, we show 
  in Corollary~\ref{Cor:Homology-Rose-S1} and Corollary~\ref{Cor:Mod-Acyclic-Fund}
    that if a finite $2$-complex $K$ is a mod-$p$ homology rose (or mod-$p$ acyclic space)
    with abelian fundamental group, then $K$
    is homotopy equivalent to $S^1$ (or a pseudo-projective plane).\n

      The paper is organized as follows. In section 2, we 
      give an elementary proof of Theorem~\ref{Thm:Covering}.
     In section 3, we investigate the properties of the fundamental groups of 
     $R$-homology roses and prove Proposition~\ref{Prop:Fund-Group} and 
     Theorem~\ref{thm:Carlsson}.
     In particular, we determine which finitely generated 
     abelian groups can be realized as 
     the fundamental groups of mod-$p$ homology roses and mod-$p$ acyclic spaces. 
     In section 4, we introduce the notion of $\F$-gap for a finitely presentable group
     and study some of its properties.  
    This notion is a generalization of  
     the efficiency of a group and is related some other well known concepts. 
     In section 5, we discuss how to 
     search candidates for the 
     counterexamples of Wall's $\mathrm{D(2)}$-problem
     among $\F$-homology roses.  Then in section 6 and section 7,
      we practice the search from the fundamental groups of 
     closed 
     $3$-manifolds and from finite groups with trivial multiplicator and negative deficiency.     
    \n

     By abuse of terminology,  
     we think of any $1$-dimensional CW-complex in this paper as a $2$-dimensional 
     CW-complex without $2$-cells. In addition, if no coefficients 
     are specified, homology and cohomology groups are with $\Z$-coefficients.\\

 \section{$(\Z_p)^r$-covering spaces of mod-$p$ homology roses}
  
    \begin{lem}\label{Lem:H_1}
     If $K$ is a mod-$p$ homology rose with $m$ petals, then 
     $H_1(K;\Z) \cong \Z^m\oplus T$ where $T$ is a finite abelian group without $p$-torsion. 
     \end{lem}
  \begin{proof}
   By the universal coefficient theorem,
     $$H_2(K;\F_p)\cong H^2(K;\F_p) 
     = \mathrm{Hom}(H_2(K);\F_p) \oplus \mathrm{Ext}(H_1(K);\F_p) $$
       Then $H_2(K;\F_p)=0$ implies that $H_1(K;\Z)$ has no $p$-torsion.
        In addition, since $H_1(K; \F_p)  = H_1(K;\Z)\otimes \F_p \cong (\F_p)^m$,
        the free rank of $H_1(K;\Z)$ must equal $m$.
  \end{proof}
  \nn
 
  Before giving a proof of Theorem~\ref{Thm:Covering}, we want to acknowledge that
 Theorem~\ref{Thm:Covering} is also a corollary of the following result in~\cite{BrownKahn77}.
  \nn
  \begin{thm}[Proposition 1.2 of~\cite{BrownKahn77}] \label{Thm:Brown}
 Let $\widetilde{B} \rightarrow B$ be a finite sheeted regular covering
  with deck transformation group $G$, where $B$ is a connected CW-complex and $\mathrm{cd}(B) < \infty$.
 Let $M$ be a $\pi_1(B)$-module and $n$ an integer such that
 $H^i(\widetilde{B}; M)=0$ for $i>n$. Then
 $H^i(B; M)=0$ for $i>n$, and $\tau: H^n(\widetilde{B}; M)_G \rightarrow H^n(B; M)$
 is an isomorphism.
 \end{thm}
  Here $\mathrm{cd}(B) := \mathrm{sup}\{ i\in \Z \, ; \, H^i(B; M) \neq 0 \ \text{for some}\ 
      \pi_1(B)\text{-module}\ M\}$, called the \emph{cohomological dimension} of $B$.
      By the work of Wall (\cite{Wall65,Wall66}), $\mathrm{cd}(B)<\infty$ implies that
      $B$ is homotopy equivalent to a CW-complex of finite dimension.\nn
      
      In addition, for any commutative ring $R$, we define
       $$\mathrm{cd}_{R}(B) := \mathrm{sup}\{ i\in \Z \, ; \, H^i(B;R) \neq 0 \},$$
       called
       the \emph{$R$-cohomological dimension} of $B$. It is clear that $\mathrm{cd}_{R}(B)\leq 
       \mathrm{cd}(B)$.  \n
       
    In the above theorem, let $M=\F_p$ (as a trivial $\pi_1(B)$-module). We obtain
    $$\mathrm{cd}_{\F_p}(\widetilde{B}) \geq \mathrm{cd}_{\F_p}(B).$$
     Moreover,      
       for any $\F_p[(\Z_p)^r]$-module $L$, 
       the co-invariant $L_{(\Z_p)^r} \neq \varnothing$ if and only if 
       $L\neq 0$ (see \cite[p.149]{Brown82}).
        So if $\widetilde{B}$ is a regular $(\Z_p)^r$-covering of $B$,
          the last sentence of Theorem~\ref{Thm:Brown} implies that
         $\mathrm{cd}_{\F_p}(\widetilde{B}) = \mathrm{cd}_{\F_p}(B)$. So we actually obtain \nn
    
     \begin{cor}
       If $\widetilde{B} \rightarrow B$ is a regular $(\Z_p)^r$-covering space with
       $\mathrm{cd}_{\F_p}(B) < \infty$, then $\mathrm{cd}_{\F_p}(\widetilde{B}) 
        = \mathrm{cd}_{\F_p}(B)$.
      \end{cor}
      
      Note that a mod-$p$ homology rose is nothing but a path-connected space $K$ with $\mathrm{cd}_{\F_p}(K)=1$.
   So Theorem~\ref{Thm:Covering} follows from the above corollary. \nn
   
  In the rest of this section,
   we give an alternative proof of Theorem~\ref{Thm:Covering} without 
  referring to~\cite[Proposition 1.2]{BrownKahn77}. Our proof involves 
  a very simple spectral sequence and some elementary facts on the modular 
  representations of $\F_p[\Z_p]$.
  \nn

 \begin{proof}[\textbf{Proof of Theorem~\ref{Thm:Covering}}]
  Since $(\Z_p)^r$ acts freely and cellularly on $X$, its orbit space
   $K=X\slash (\Z_p)^r$ is also a finite CW-complex. 
   It is easy to see that there is a 
       sequence $X=X_r \rightarrow X_{r-1} \rightarrow \cdots \rightarrow X_1 \rightarrow X_0 = K$
       where $X_j \rightarrow X_{j-1}$ is a regular $\Z_p$-covering for each $1\leq j \leq r$.
       So it is sufficient to prove the theorem for $r=1$.
       In the rest of the proof, we assume that
        $\xi: X \rightarrow K$ is a regular $\Z_p$-covering. 
  \nn
  
  (1) Assume that $K$ is a mod-$p$ homology rose with $m$ petals, we want 
   to show that $X$ is also a mod-$p$ homology rose. By lemma~\ref{Lem:H_1}, 
   $H_1(K) \cong \Z^m\oplus T$ where $T$ is a finite abelian group without $p$-torsion.
   Let $\widetilde{K}$ be the regular covering of $K$ with
   $\pi_1(\widetilde{K})$ isomorphic to the kernel of the following group epimorphism
   $$ \widetilde{\sigma}: \pi_1(K) \longrightarrow H_1(K) = 
   \Z^m\oplus T \longrightarrow \Z^m. $$
   In other words, $\widetilde{K}$ is the maximal free abelian covering of $K$. 
   Meanwhile, $\pi_1(X)$ is isomorphic to the kernel of 
   a group epimorphism $\sigma_X: \pi_1(K) \rightarrow \Z_p$. Since
   $T$ has no $p$-torsion, $\sigma_X$ can factor through 
   $\widetilde{\sigma}$, i.e. there exists a group epimorphism 
   $\eta_X: \Z^m \rightarrow \Z_p$ so that
   $$\sigma_X = \eta_X \circ\widetilde{\sigma}: \pi_1(K) \longrightarrow 
    \Z^m \overset{\eta_X}{\longrightarrow} \Z_p.$$
   So $\widetilde{K}$ is also a regular covering of $X$. Let 
   $\mathcal{D}(\widetilde{K})$ and $\mathcal{D}(X)$ be the
   deck transformation group of $\widetilde{K}$ and $X$ over $K$, respectively.
   It is clear that 
    $$\mathcal{D}(\widetilde{K}) \cong \Z^m,\ \mathcal{D}(X)\cong \Z_p.$$
      Let $H$ be the normal subgroup of $\mathcal{D}(\widetilde{K})$ so that
      $\mathcal{D}(\widetilde{K})\slash H \cong \mathcal{D}(X)$. 
       So $H$ is a free abelian group of rank $m$ which 
       can be identified with the kernel of $\eta_X$. Then we have a short exact sequence
       of abelian groups
       \[ 0\longrightarrow H \cong \Z^m \overset{\Lambda_X}{\longrightarrow} \Z^m \cong \mathcal{D}(\widetilde{K}) 
        \overset{\eta_X}{\longrightarrow} \Z_p \longrightarrow 0. \]
     The map $\Lambda_X$ can be represented by a diagonal matrix 
     $\mathrm{Diag}(1,\cdots, 1, p)$ with respect to a 
     properly chosen basis of $\mathcal{D}(\widetilde{K})$, say $\{ t_1,\cdots, t_{m-1}, t_m\}$.
     Then $H$ is generated by $\{ t_1,\cdots, t_{m-1}, pt_m \}$. Let 
     $$ H' := \langle t_1,\cdots, t_{m-1}\rangle \subset \mathcal{D}(\widetilde{K}),\
      H\slash H' \cong \Z. $$
   
      Then $X'=\widetilde{K}\slash H'$ is an regular $\Z$-covering of $K$ and, at the same time,
      $X'$ is a regular $\Z$-covering of $X$. Indeed, we have
       $$X= \widetilde{K} \slash H = (\widetilde{K} \slash H')\slash (H'\slash H) 
       = X'\slash (H'\slash H).$$
       From the covering $X'$ over $K$, we obtain the following short exact sequence
        of chain complexes
       \begin{equation} \label{Equ:K_infty}
            0 \longrightarrow C_*(X';\F_p)
            \overset{t_m-1}{\xrightarrow{\hspace*{0.7cm}}}
             C_*(X';\F_p) \longrightarrow C_*(K;\F_p)
             \longrightarrow 0.
        \end{equation}
   By the homology long exact sequence of~\eqref{Equ:K_infty} and the assumption that
   $K$ is a mod-$p$ homology rose, we can conclude that
   $H_i(X';\F_p)  \overset{t_m-1}{\xrightarrow{\hspace*{0.7cm}}} H_i(X';\F_p)$
   is an isomorphism for any $i\geq 2$ and an injection for $i=1$. 
      \nn
      
            From the covering $X'$ over $X$, we have the short exact sequence
       \begin{equation} \label{Equ:K_infty-2}
            0 \longrightarrow C_*(X';\F_p)
            \overset{t_m^p-1}{\xrightarrow{\hspace*{0.7cm}}} C_*(X';\F_p) \longrightarrow
            C_*(X;\F_p) \longrightarrow 0.
        \end{equation}
                
       Notice that, over $\F_p$, $(t_m^p-1) = (t_m-1)^p$. So from
       the homology long exact sequence of~\eqref{Equ:K_infty-2}, we similarly obtain that
     $H_i(X';\F_p)  \overset{t^p_m-1}{\xrightarrow{\hspace*{0.7cm}}} H_i(X';\F_p)$ 
     is an isomorphism for any $i\geq 2$ and an injection for $i=1$. This implies that
            $H_i(X;\F_p)=0,\ i\geq 2$, i.e. $X$ is a mod-$p$ homology rose.\nn

      Moreover, since $K$ is a finite CW-complex, the Euler
         characteristics of $X$ and $K$ satisfy the relation $\chi(X) = p \cdot \chi(K)$. This
         implies that $X$ is a mod-$p$ homology rose with $(m-1)p +1$ petals. \nn
         
       (2) Conversely, assume $X$ is a mod-$p$ homology rose. 
               The covering canonically determines a fibration
       \begin{equation} \label{Equ:Serre-Fibration}
          X\longhookrightarrow X_{\Z_p} \longrightarrow B\Z_p
        \end{equation}  
        where $X_{\Z_p} = E\Z_p \times X\slash \sim$ is the Borel construction of
       the $\Z_p$-action on $X$ and $B\Z_p$ is the classifying space of $\Z_p$ 
       (see~\cite{AllPupp93}).
       Since the action of $\Z_p$ on $X$ is free, $X_{\Z_p} \simeq X\slash \Z_p = K$. 
       The cohomology Laray-Serre spectral sequence $(E_*,d_*)$ (with $\F_p$-coefficients) 
       of the fibration~\eqref{Equ:Serre-Fibration}
       converges to $H^*(X_{\F_p};\F_p) = H^*(K;\F_p)$ whose $E_2$-term is
       \[  E_2^{j,k} =  H^j(B\Z_p ;  \mathcal{H}^k(X;\F_p)) = 
            H^j(\Z_p ;  \mathcal{H}^k(X;\F_p))    \] 
        where $\mathcal{H}^k(X;\F_p)$ denotes $H^k(X;\F_p)$ as a
        $\F_p[\Z_p]$-module (see~\cite[Chapter 5]{SpecSeq2001}).
       Here $\F_p$-coefficient is implicitly assumed in all the cohomology groups.
        \nn
         Since $X$ is a mod-$p$ homology rose, 
        $ E_2^{j,k} = 0$ if $k\neq 0,1$.
        For convenience, let
         $$G_j := E^{j,1}_2 = H^j(\Z_p;  \mathcal{H}^1(X;\F_p)).$$   
        The differential $d_2 : E_2^{j,k} \rightarrow E_2^{j+2, k-1}$ is as shown in 
        Figure~\ref{Diagram:SpecSeq}.
  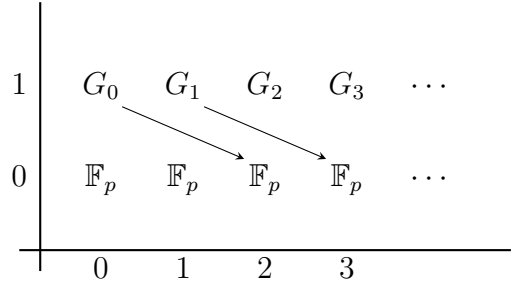
\begin{figure}
    \begin{tikzpicture} 
      \matrix (m) [matrix of math nodes,
    nodes in empty cells,nodes={minimum width=5ex,
    minimum height=5ex,outer sep=-5pt},
    column sep=1ex,row sep=1ex]{
                &      &     &     &  &  & \\
          1     &  G_0 &  G_1  & G_2 & G_3 & \cdots & \\
          0     &  \F_p  & \F_p &  \F_p  & \F_p & \cdots & \\
    \quad\strut &   0  &  1  &  2  & 3 &  & \strut \\};
    \draw[-stealth] (m-2-2.south east) -- (m-3-4.north west);
    \draw[-stealth] (m-2-3.south east) -- (m-3-5.north west);
    \draw[thick] (m-1-1.east) -- (m-4-1.east) ;
   \draw[thick] (m-4-1.north) -- (m-4-7.north) ;
   \end{tikzpicture}  
    \caption{The $E_2$-term}\label{Diagram:SpecSeq}
  \end{figure}   
   To compute $G_j$, we need to know the $\F_p[\Z_p]$-module structure
   of $\mathcal{H}^1(X;\F_p)$. \nn
   
   In the rest of the proof, we identify $\F_p[\Z_p]$ with $\F_p[x]\slash x^p$ where $x=t-1$ and
   $t$ is a generator of $\Z_p$. Then
   by the theory of modular representation of finite groups (see~\cite{Alperin86}),
   any indecomposable $\F_p[\Z_p]$-module over $\F_p$ is isomorphic to
   $\mathcal{R}_k = \F_p[a]\slash a^k , 1\leq k \leq p$ where  
   $x\in \F_p[x]\slash x^p$ 
   acts on $\mathcal{R}_k$ by multiplying each element of $\mathcal{R}_k$ by $a$.
   If we choose $\{ 1, a,\cdots, a^{k-1} \}$ as a linear basis of $\mathcal{R}_k$, 
   then
   \begin{equation}\label{Equ:x-action}
       x\cdot (1, a,\cdots, a^{k-1}) = (1, a,\cdots, a^{k-1}) 
           \begin{pmatrix}
             0 & 0 & \cdots & 0 & 0   \\
             1 & 0 & \cdots & 0 & 0\\
             0 & 1 &  \cdots & 0 & 0 \\
             \vdots &\vdots &\vdots &\vdots & \vdots\\
            0 & 0 &  \cdots & 1 & 0 
       \end{pmatrix} .
     \end{equation}
   Notice that $\mathcal{R}_1 = \F_p$ is the trivial $\F_p[\Z_p]$-module.
   Let $N_k$ denote the $k\times k$ matrix on the right hand side of~\eqref{Equ:x-action}. 
   It is clear that $\mathrm{rank}_{\F_p}(N_k) = k-1$.
       Suppose 
   \[ \mathcal{H}^1(X;\F_p) \cong \mathcal{R}_1^{l_1} \oplus \cdots \oplus \mathcal{R}_p^{l_p},
     \ \, l_1,\cdots, l_p \geq 0.\]
  Then by the Lemma~\ref{Lem:Homology-TwistCoeff} below,
  \begin{equation} \label{Equ:G_j}
     G_j \cong \begin{cases}
       (\F_p)^{\sum^{p}_{i=1} l_i},   &   \text{ $j=0$; } \\
       (\F_p)^{\sum^{p-1}_{i=1} l_i},     &   \text{ $j\geq 1$.} \\
 \end{cases}
   \end{equation}
   Let us see how to determine $l_1,\cdots, l_p$ from our conditions.
    Suppose $X$ is mod-$p$ homology rose with $n$-petals. Then
    \begin{equation} \label{Equ:n-equality-1}
     \mathrm{rank}_{\F_p}(\mathcal{H}^1(X;\F_p)) =   n=l_1 + 2\cdot l_2 +\cdots + p\cdot l_p 
     \end{equation}     
  In addition, the relation of Euler characteristics $\chi(X)=p\cdot\chi(K)$ implies
   \begin{equation} \label{Equ:n-equality-2}
       n = 1-p\cdot \chi(K).
     \end{equation}  
  \textbf{Claim:} $l_1=1$, $l_2=\cdots = l_{p-1} =0$, $l_p= -\chi(K)$.\nn
  
 Indeed, the above spectral sequence converges to $H^*(K;\F_p)$ which is nontrivial
       in only finite dimensions. Then since $G_j$, $j\geq 1$ are all isomorphic, we must have
       \begin{equation} \label{Equ:sum-of-di}
          \mathrm{rank}_{\Z_p}(G_j) = l_1 + \cdots + l_{p-1}  = 1, \ \forall\, j\geq 1.
       \end{equation}   
       This is because if $l_1 + \cdots + l_{p-1}  = 0$, every $E_2^{j,0}$ ($j\geq 3$) 
        will survive to $E_{\infty}$-term which contradicts the homological finiteness of $K$
        (see Figure~\ref{Diagram:SpecSeq}).
        Similarly, if $l_1 + \cdots + l_{p-1}  > 1$, $G_j=E_2^{j,1}$ ($j\geq 0$) can not be
         killed by $d_2$, which again contradicts the homological finiteness of $K$.\nn
         
        By~\eqref{Equ:sum-of-di}, exactly one of $l_1,\cdots, l_{p-1}$ is equal to one, 
        others are all zero. If we assume $l_k=1$ for some $1\leq k \leq p-1$,
              then $n=k + pl_p$ by~\eqref{Equ:n-equality-1}. 
              By plugging this into~\eqref{Equ:n-equality-2}, we get
         \[ k + p\cdot l_p = 1-p\cdot \chi(K) \ \Longrightarrow  \  
         p \cdot (l_p + \chi(K))=1-k.  \]
        Since $1\leq k \leq p-1$, 
        the only possibility is $k=1$ and $l_p=-\chi(K)$. The claim is proved.   \nn

      Therefore, $\mathcal{H}^1(X;\F_p) = \mathcal{R}_1 \oplus 
      \mathcal{R}_p^{-\chi(K)}$ and so
      \[    E_2^{j,1} =   H^j(\Z_p;  \mathcal{H}^1(X;\F_p)) = H^j(\Z_p; \mathcal{R}_1 \oplus
         \mathcal{R}_p^{-\chi(K)} ) = H^j(\Z_p; \mathcal{R}_1) \oplus 
          H^j(\Z_p; \mathcal{R}_p^{-\chi(K)})  \]
      
   Next, let us see what $d_2:  E_2^{j,1} \rightarrow  E_2^{j+2,0}$ should be
  when restricted to the $H^j(\Z_p; \mathcal{R}_1)\cong \F_p \subset E_2^{j,1}$
   (see Figure~\ref{Diagram:SpecSeq-2}). In this case, since $\mathcal{R}_1=\F_p$ is 
   the trivial $\F_p[\Z_p]$-module, the  
  multiplicative structure of the spectral sequence implies that 
  $d_2|_{H^j(\Z_p; \mathcal{R}_1)}$ is determined by $d_2|_{H^0(\Z_p; \mathcal{R}_1)}$. 
  Then we claim that $d_2$ must map 
   the $H^0(\Z_p; \mathcal{R}_1) \subset E_2^{0,1}$ 
   isomorphically onto $E_2^{2,0} \cong \F_p$, because otherwise
   all the $H^j(\Z_p; \mathcal{R}_1) \subset E_2^{j,1}$ $(j\geq 0)$ will survive
   to $E_{\infty}$, which contradicts the homological finiteness of $K$.
  \n
  
   Moreover, 
     $H^j(\Z_p; \mathcal{R}_p^{-\chi(K)})=0$ for all $j\geq 1$, hence has no contribution to
     $H^{\geq 2}(K;\F_p)$. So we obtain that $H^i(K;\F_p)=0$ for all $i\geq 2$, i.e. 
     $K$ is mod-$p$ homology rose.   
     This finishes the proof of Theorem~\ref{Thm:Covering}. 
  \end{proof}
  
    \begin{figure}
    \begin{tikzpicture} 
      \matrix (m) [matrix of math nodes,
    nodes in empty cells,nodes={minimum width=5ex,
    minimum height=5ex,outer sep=-5pt},
    column sep=1ex,row sep=1ex]{
                &      &     &     &  &  & \\
          1     &  \F_p &  \F_p  & \F_p & \F_p & \cdots & \\
          0     &  \F_p  & \F_p &  \F_p  & \F_p & \cdots & \\
    \quad\strut &   0  &  1  &  2  & 3 &  & \strut \\};
    \draw[-stealth] (m-2-2.south east) -- (m-3-4.north west);
    \draw[-stealth] (m-2-3.south east) -- (m-3-5.north west);
    \draw[thick] (m-1-1.east) -- (m-4-1.east) ;
   \draw[thick] (m-4-1.north) -- (m-4-7.north) ;
   \end{tikzpicture}  
    \caption{}\label{Diagram:SpecSeq-2}
  \end{figure}

 \begin{lem} \label{Lem:Homology-TwistCoeff}
  For any $1\leq k \leq p-1$, $H^j(\Z_p; \mathcal{R}_k) = \F_p$ for all $j\geq 0$.
  For $k=p$,
  \[ H^j(\Z_p; \mathcal{R}_p) =  \begin{cases}
    \F_p,   &   \text{ $j = 0$, } \\
     0     &   \text{ $j\neq 0$. } 
       \end{cases}
  \] 
 \end{lem}
 \begin{proof}
  If we identity $\F_p[\Z_p]$ with $\F_p[x]\slash x^p$, 
   a projective resolution of the trivial $\F_p[\Z_p]$-module $\F_p$ is:
  \[ \cdots\ \,  \F_p[\Z_p] \overset{\cdot x}{\xrightarrow{\hspace*{0.6cm}}} \F_p[\Z_p] 
     \overset{\cdot\, x^{p-1}}{\xrightarrow{\hspace*{0.6cm}}} \F_p[\Z_p]
     \overset{\cdot x}{\xrightarrow{\hspace*{0.6cm}}} \F_p[\Z_p] \overset{\varepsilon}{\xrightarrow{\hspace*{0.6cm}}} \F_p  
  \]
 where $\varepsilon$ is a ring homomorphism defined by 
  $\varepsilon(x)=0$ and $\varepsilon(1)=1$. Then 
  $H^j(\Z_p; \mathcal{R}_k)$ is the $j$-th cohomology group of the following chain complex
   \[ \cdots\ \, 
   \mathrm{Hom}_{\F_p[\Z_p]}(\F_p[\Z_p], \mathcal{R}_k) \overset{\delta_2}{\xleftarrow{\hspace*{0.6cm}}} \mathrm{Hom}_{\F_p[\Z_p]}(\F_p[\Z_p], \mathcal{R}_k)
    \overset{\delta_1}
  {\xleftarrow{\hspace*{0.6cm}}} \mathrm{Hom}_{\F_p[\Z_p]}(\F_p[\Z_p], \mathcal{R}_k)
     \]
     \[ \qquad\qquad \qquad\qquad \qquad\qquad \qquad\qquad \qquad\qquad 
     \overset{\delta_0}{\xleftarrow{\hspace*{0.6cm}}} 
     \mathrm{Hom}_{\F_p[\Z_p]}(\F_p[\Z_p], \mathcal{R}_k) \]
 where $\delta_{even} = \mathrm{Hom}_{\F_p[\Z_p]}(\cdot x)$,
 $\delta_{odd} = \mathrm{Hom}_{\F_p[\Z_p]}(\cdot x^{p-1})$.\n
 
  By identifying $\mathrm{Hom}_{\F_p[\Z_p]}(\F_p[\Z_p], \mathcal{R}_k)$ with $(\F_p)^k$ as 
 a module over $\F_p$, we can write the above chain complex explicitly as
   \begin{equation} \label{Equ:Seq-R_k}
   \cdots\ \,  (\F_p)^k \overset{N^t_k}{\xleftarrow{\hspace*{1cm}}} (\F_p)^k 
     \overset{\ (N_k^t)^{p-1}}{\xleftarrow{\hspace*{1cm}}} (\F_p)^k 
          \overset{N^t_k}{\xleftarrow{\hspace*{1cm}}} (\F_p)^k  
     \end{equation}
 where $ N_k^t$ is the transpose of the matrix $N_k$.  
 Then our lemma follows from the fact that $\mathrm{rank}_{\F_p}(N_k) = k-1$ and
 $$\mathrm{rank}_{\F_p} (N_k^{p-1}) = \begin{cases}
  0,   &   \text{ $1\leq k \leq p-1$; } \\
   1,     &   \text{ $k=p$. } \\
 \end{cases}$$

 \end{proof}
  \nn

\section{Fundamental Groups of Homology Roses} \label{Sec3}
 First of all, let us prove Proposition~\ref{Prop:Fund-Group}.
 \nn
 
 \begin{proof}[\textbf{Proof of Proposition~\ref{Prop:Fund-Group}}]
   The necessity follows from Lemma~\ref{Lem:H_1}. 
  For the sufficiency, let us first choose a connected $2$-complex $K$ with
  $\pi_1(K)=\Gamma$. By
  the following short exact sequence due to Hopf (see~\cite[Ch.II $\S 5$]{Brown82}) 
    \begin{equation} \label{Equ:Hopf-1}
        \pi_2(K) \longrightarrow H_2(K) \longrightarrow H_2(\Gamma) \longrightarrow 0, 
    \end{equation}    
    we get a new short exact sequence
    \begin{equation}  \label{Equ:Hopf-2}
         \pi_2(K)\otimes R \longrightarrow H_2(K)\otimes R
    \longrightarrow H_2(\Gamma)\otimes R \longrightarrow 0. 
   \end{equation} 
   The assumption $H_2(\Gamma;R)=0$ implies 
   $$\mathrm{Tor}(H_1(K);R)=\mathrm{Tor}(H_1(\Gamma);R)=0,\ \
   H_2(\Gamma)\otimes R = 0.$$ 
   So
    $H_2(K;R) = H_2(K)\otimes R$.
   Then from Equation~\eqref{Equ:Hopf-2}, we obtain a surjective 
   map 
   $ \pi_2(K)\otimes R \longrightarrow H_2(K;R)$. This implies that any
    $\alpha\in H_2(K;R)$ can be represented by a continuous map
   $\varphi_{\alpha}: S^2 \rightarrow K$. Since $H_2(K)$ is clearly a free abelian group,
   $H_2(K;R)$ is a free $R$-module.
    Suppose $\{ \alpha_i \}_{i\in I}$ is a 
   a set of generators of $H_2(K;R)$ over $R$. Then for all $i\in I$, we glue
   a $3$-ball $D^3$ to $K$ via $\varphi_{\alpha_i}$, which gives us
    a $3$-complex
   $K'$. It is easy to check that $K'$ is an $R$-homology rose with $\pi_1(K')\cong \Gamma$.
   Note that if $\Gamma$ is finitely presentable,
      $K'$ can be taken to be a finite $3$-complex.  
 \end{proof}

     Note that the same argument as the proof of 
    Proposition~\ref{Prop:Fund-Group} can be used to
      prove the following proposition.\n     
     
     \begin{prop} \label{Prop:mod-p-Acyclic}
    For any commutative ring $R$ with a unit, a group $\Gamma$
     can be realized as the fundamental group of an $R$-acyclic space
      if and only if
     $H_1(\Gamma;R)= H_2(\Gamma;R)=0$.
    \end{prop}
    \n
    
    \begin{proof}[\textbf{Proof of Theorem~\ref{thm:Carlsson}}]
     Let $\Gamma=\pi_1(B)$ and let
     $b_1(\Gamma;\F_p) = b_1(B;\F_p)=m$. Then 
     since $X$ is a regular $(\Z_p)^r$-covering over $B$, we must have $r\leq m$.
         \n
           If $m=0$, our claim clearly holds.           
            If $m\geq 1$, 
      Proposition~\ref{Prop:Fund-Group} tells us that
      there exists a finite CW-complex $B'$ with $\pi_1(B')\cong \Gamma$ and
      $B'$ is a mod-$p$ homology rose with $m$ petals.
      In addition, the fundamental group $\pi_1(X)=N$
      can be identified with a normal subgroup of $\Gamma$ so that
       $\Gamma\slash N \cong (\Z_p)^r$. Then we can also think of $N$ as a 
       subgroup of $\pi_1(B')$, which determines a regular $(\Z_p)^r$-covering $X'$ over $B'$ with $\pi_1(X')\cong N$. 
       By Theorem~\ref{Thm:Covering}, $X'$ is also a mod-$p$ homology rose.
       Then from the fact that $b_1(B';\F_p)=b_1(\Gamma;\F_p)=m$ and
        the Euler characteristic $\chi(X')= p^r\cdot \chi(B')$,
       we obtain that
       $$ \ \ \ b_1(X;\F_p) = b_1(N;\F_p) = b_1(X';\F_p) = p^r(m-1) +1 \geq p^r(r-1)+1.$$        
     \begin{itemize}
       \item When $r=1$, $ b_1(X;\F_p) \geq 1 =2^r-1$;
                        \n
       \item When $r\geq 2$, $ b_1(X;\F_p) \geq p^r+ 1 > 2^r -1$.
     \end{itemize}
     
       So the theorem is proved. 
    \end{proof}
    \n
    
    \begin{proof}[\textbf{Proof of Corollary~\ref{Cor:Carlsson}}]
    Note that $H_2(B;\F_p)=0$ implies $H_2(B)\otimes \F_p =0$ and 
      $\mathrm{Tor}(H_1(B);\F_p)=0$. Then by~\eqref{Equ:Hopf-2},
       $H_2(\pi_1(B);\F_p)=0$. So by  
    Theorem~\ref{thm:Carlsson}, we obtain that
    $\sum^{\infty}_{i=0} b_i(X;\F_p) \geq b_0(X;\F_p) + b_1(X;\F_p) \geq 1 + (2^r-1) =2^r$.
    \end{proof}
  \n
  
   Suppose $G$ is a finitely presentable group. Let
     $\mathcal{P} = \langle a_1\cdots, a_n \,| \, r_1,\cdots, r_m \rangle $ be
      a finite presentation of $G$. The integer $n-m$ is called the \emph{deficiency} of 
      $\mathcal{P}$, denoted by $\mathrm{def}(\mathcal{P})$.
     The \emph{deficiency} of $G$, denoted by
      $\mathrm{def}(G)$,
      is the maximum over all its finite presentations, of the deficiency of each presentation.
      Note that if $G$ is not finitely presentable, we can still define 
      the presentation complex of any presentation of $G$, but the notion 
      of deficiency does not make sense anymore.\nn

      Any presentation $\mathcal{P}$ canonically determines
      a $2$-dimensional CW-complex $K_{\mathcal{P}}$ called the
      \emph{presentation complex} of $\mathcal{P}$.
         \nn
     \begin{itemize}
       \item $K_{\mathcal{P}}$ has a single vertex $q_0$,
       and one oriented $1$-cell $\gamma_j$ attached to $q_0$ for each generator $a_j$
    ($1\leq j \leq n$). So the $1$-skeleton of $K_{\mathcal{P}}$ is a bouquet of $n$ circles attached to $q_0$.\nn

       \item $K_{\mathcal{P}}$ has one oriented $2$-cell $\beta_i$ for each 
       relator $r_i$ ($1\leq i \leq m$), where $\beta_i$ is attached to the $1$-skeleton of 
       $K_{\mathcal{P}}$ via a map defined by $r_i$.\nn
     \end{itemize}

      \begin{lem} \label{Lem:Deficiency}
         For any finitely presentable group $G$ and any field $\F$,
         $$ {\rm def}(G) \leq b_1(G;\F) - b_2(G;\F).$$
      \end{lem}
      \begin{proof}
           For any finite presentation 
           $\mathcal{P} = \langle a_1\cdots, a_n \,| \, r_1,\cdots, r_m \rangle $
           of $G$, there is an associated resolution of $\Z$ by projective $\Z G$-modules
           of the following form:
           \[   \cdots \rightarrow \Z [G]^{\oplus m} \rightarrow  \Z[G]^{\oplus n} \rightarrow
           \Z[G] \rightarrow \Z \rightarrow 0 \]
      Indeed, this comes from the cellular decomposition of the universal covering space of the
      $K(G,1)$ space built from the presentation complex $K_{\mathcal{P}}$ of $\mathcal{P}$.
      Then applying the functor $\otimes_{\Z[G]} \F$ to this projective resolution, we get
      a chain complex:
      \[  \cdots \rightarrow \Z [G]^{\oplus m} \otimes_{\Z[G]} \F \rightarrow  
             \Z[G]^{\oplus n}\otimes_{\Z[G]} \F \rightarrow
           \Z[G]\otimes_{\Z[G]} \F \rightarrow  \F \rightarrow 0\]
     whose homology groups are just $H_*(G;\F)$ (see~\cite{Brown82}).
      Moreover, by the Morse inequality of this chain complex, we obtain that
      \[ 1-n +m \geq b_0(G;\F) - b_1(G;\F) + b_2(G;\F). \]
      So $ n -m \leq b_1(G;\F) - b_2(G;\F)$. Since this argument works for
      arbitrary finite presentations of $G$, so we get
      $ {\rm def}(G) \leq b_1(G;\F) - b_2(G;\F)$.
       \end{proof}

   \begin{prop} \label{prop:Deficiency-1}
    If a finite $2$-complex $K$ is an 
    $\F$-homology rose with $m$ petals, then the deficiency of the 
    fundamental group $\pi_1(K)$ of 
    $K$ is equal to $m$. If $K$ is an $\F$-acyclic space, 
    then the deficiency of $\pi_1(K)$ is equal to $0$.
  \end{prop}
  \begin{proof}
    Without loss of generality, we can assume that $K$ has a single $0$-cell and $n$ $1$-cells. 
    Then the number of $2$-cells in $K$ is $n-m$.
    So $\mathrm{def}(\pi_1(K)) \geq m$. On the other hand,
    since $b_1(\pi_1(K);\F)=m$ and $b_2(\pi_1(K);\F)=0$,
    Lemma~\ref{Lem:Deficiency} implies that
    $\mathrm{def}(\pi_1(K))\leq m$. So we must have $\mathrm{def}(\pi_1(K)) = m$.
    The same argument clearly works for mod-$p$ acyclic spaces.
  \end{proof}
  
 It is clear that any group with positive deficiency must be infinite. 
 So we obtain the following corollary.\n
   
  \begin{cor}
    If a finite $2$-complex $K$ is an 
    $\F$-homology rose, $\pi_1(K)$ must be infinite.
  \end{cor}

  Next, we investigate which finitely generated 
  abelian groups can be realized as the fundamental groups of
     mod-$p$ homology roses.\nn
    
  \begin{lem} \label{Lem:Deficiency-Abelian}
    For a finitely generated abelian group
    $A=\Z^{r} \oplus \Z_{d_1}\oplus \Z_{d_2}\oplus\cdots \oplus \Z_{d_k}$ where $d_1>1$ and
    $d_1 | d_2 | \cdots | d_k$,
    $\mathrm{def}(A) = r - \binom{r+k}{2}$.   
  \end{lem} 
  \begin{proof}
    It is clear that $A$ can be presented by 
    \begin{equation} \label{Equ:Presentation-Abelian}
       \mathcal{P}_A=\langle \, a_1,\cdots, a_{r}, b_1,\cdots, b_k \ |\ b^{d_1}_1,\cdots, 
          b^{d_k}_k,\, [a_i,a_{i'}], [b_j,b_{j'}], [a_i,b_j], \qquad\ \ 
     \end{equation}
     $$\qquad \qquad\qquad\qquad\ \ 1\leq i < i' \leq r, 1\leq j < j'\leq k    
        \, \rangle.$$ 
        So $\mathrm{def}(A)\geq \mathrm{def}(\mathcal{P}_A)= r - \binom{r+k}{2}$.
     On the other hand, for a finitely presented group $G$, it is well known that (see~\cite[14.1.5]{Robin96})
     \begin{equation} \label{Def:Deficiency-G}
        \mathrm{def}(G) \leq r(G_{ab}) - d(M(G)) 
     \end{equation}   
     where $r(G_{ab})$ is the free rank of the abelianization $G_{ab}$ of $G$ (i.e.
     $r(G_{ab}) = b_1(G;\Q)$), 
     $M(G)\cong H_2(G;\Z)$ is the
     \emph{Schur multiplicator} of $G$ and
     $d(M(G))$ is the minimum number of elements that can generate $M(G)$.
    For our group $A$, 
      \begin{equation} \label{Equ:H2-Abelian}
           M(A)= H_2(A;\Z) \cong \Z^{\binom{r}{2}}\oplus \Z^{r+k-1}_{d_1} \oplus 
            \Z^{r+k-2}_{d_2}\oplus  \cdots \oplus \Z^r_{d_k}.
      \end{equation}        
   The number of elements that can generate $M(A)$ is at least 
    $$ \binom{r}{2} + (r+k-1) + (r+k-2) + \cdots + r = \binom{r+k}{2}.$$  
    This is because if choose a prime $p|d_1$, then $M(A)\slash pM(A)$ will be a 
    vector space over $\F_p$ of dimension $\binom{r+k}{2}$, which 
    requires at least $\binom{r+k}{2}$ generators for $M(A)$.
   So 
    \begin{equation}\label{Equ:d(M(A))}
       d(M(A)) = \binom{r+k}{2}.
     \end{equation}  
   Then by~\eqref{Def:Deficiency-G}, we have
    $ \mathrm{def}(A) \leq r(A) - d(M(A)) = r - \binom{r+k}{2}$. The lemma is proved.    
  \end{proof}
    
   Suppose an abelian group 
   $A=\Z^{r} \oplus \Z_{d_1}\oplus \Z_{d_2}\oplus\cdots \oplus \Z_{d_k}$  
   where $d_1>1$ and $d_1 | d_2 | \cdots | d_k$ 
   is the fundamental group of a mod-$p$ homology rose $K$.
   Then $A \cong H_1(K)$ has no $p$-torsion and $H_2(A;\F_p)=0$ 
   by Proposition~\ref{Prop:Fund-Group}.
     So by~\eqref{Equ:H2-Abelian} and the fact that $A \cong H_1(K)$ is an 
       infinite group (by Lemma~\ref{Lem:H_1}), we must have
            \begin{equation}\label{Equ:Fund-Homology-Rose}
         \text{$r =1$ and $p\nmid d_i$, $1\leq i \leq k$.} 
      \end{equation}   
     
    Conversely, Proposition~\ref{Prop:Fund-Group} implies that 
    any finitely generated abelian group $A$ with free rank $1$ and without 
     $p$-torsion can be realized as the fundamental group of some mod-$p$ homology rose.
    Indeed, there is a more direct way to see this fact. 
    let $\mathbb{P}_m$ be the mapping cone of a map $f: S^1 \rightarrow S^1$ with degree $m$
    (called a \emph{pseudo-projective plane} of order $m$). For any $p\nmid m$, $\mathbb{P}_m$
    is a mod-$p$ acyclic space. 
    We can write $A=\Z\oplus\Z_{d_1}\oplus \Z_{d_2} \oplus \cdots \oplus \Z_{d_k}$
   where $p\nmid d_i$, $1\leq i \leq k$. Define
    $$L_A = S^1 \times \mathbb{P}_{d_1}\times \mathbb{P}_{d_2}\times 
    \cdots \times \mathbb{P}_{d_k}.$$ 
    It is clear that $L_A$ is a mod-$p$ homology circle whose fundamental group
       is $A$. So we obtain the 
       following result.
       \nn
       
     \begin{prop} \label{prop:Homology-Rose-Abelian}
        A finitely generated abelian group $A$ can be realized as the fundamental group 
        of a mod-$p$ homology rose if and only if 
        $A$ has free rank $1$ and has no $p$-torsion. 
     \end{prop}

      \begin{cor}
         If the fundamental group of a mod-$p$ homology rose $K$ is a finitely 
         generated abelian  group, $K$ must be a mod-$p$ homology circle.
      \end{cor}
      \begin{proof}
         By Proposition~\ref{prop:Homology-Rose-Abelian}, $\pi_1(K)$ must 
       be an abelian group with free rank $1$ and has no $p$-torsion.
         Then $b_1(\pi_1(K);\F_p)=1$, which implies that $K$ is a mod-$p$ homology circle.
      \end{proof}
      \n

      \begin{cor} \label{Cor:Homology-Rose-S1}
      A $2$-dimensional finite complex $K$ is a
       mod-$p$ homology rose with abelian fundamental group
      if and only if $K$ is homotopy equivalent to $S^1$.
    \end{cor}
    \begin{proof}
     The sufficiency is trivial. For the necessity, let $K$ be a mod-$p$ homology
     rose with $m$-petals. By 
     Proposition~\ref{prop:Homology-Rose-Abelian}, we 
     can assume that
    $\pi_1(K)= H_1(K) = \Z \oplus \Z_{d_1}\oplus \Z_{d_2}\oplus\cdots \oplus \Z_{d_k}$ where $d_1>1$, $d_1 | d_2 | \cdots | d_k$ and $p\nmid d_i$, $1\leq i \leq k$.
    Then since $K$ is $2$-dimensional, 
    Proposition~\ref{prop:Deficiency-1} 
     and Lemma~\ref{Lem:Deficiency-Abelian}
     imply that 
       \[ 1-\binom{k+1}{2} =  \mathrm{def}(\pi_1(K)) = m.\] 
       This forces
    $m=1$ and $k=0$, and hence $\pi_1(K)\cong \Z$. 
    Finally, Theorem~\ref{Thm:Wall} below implies that
       $K$ must be homotopy equivalent to $S^1$.
    \end{proof}
    
     \begin{thm}[see~\cite{Wall65} Proposition 3.3] \label{Thm:Wall}
    Every compact, connected $2$-complex with free fundamental
  group is homotopy equivalent to a finite bouquet of $1$- and
  $2$-dimensional spheres.
  \end{thm}
   \n
   
  For mod-$p$ acyclic spaces, we can easily obtain the following result from
  Proposition~\ref{Prop:mod-p-Acyclic} and Lemma~\ref{Lem:Deficiency-Abelian}.\nn
    
   \begin{prop} \label{Prop:Acyclic-1}
     A finitely generated abelian group $A$ can be realized as the fundamental group 
        of a mod-$p$ acyclic space if and only if $A$ is finite and has no $p$-torsion.
   \end{prop}

   \begin{cor} \label{Cor:Mod-Acyclic-Fund}
      A $2$-dimensional finite complex $K$ is a
       mod-$p$ acyclic space with abelian fundamental group
      if and only if $K$ is homotopy equivalent to a pseudo-projective plane $\mathbb{P}_m$
      with $p \nmid m$.
    \end{cor}
    \begin{proof}
      The sufficiency is trivial. For the necessity, by 
     Proposition~\ref{prop:Homology-Rose-Abelian} we 
     can assume that
    $\pi_1(K)= H_1(K) =  \Z_{d_1}\oplus \Z_{d_2}\oplus\cdots \oplus \Z_{d_k}$ where
     $d_1>1$, $d_1 | d_2 | \cdots | d_k$ and $p\nmid d_i$, $1\leq i \leq k$.
     Then since $K$ is $2$-dimensional, 
    Proposition~\ref{prop:Deficiency-1} 
     and Lemma~\ref{Lem:Deficiency-Abelian}
     imply that 
       \[ \binom{k}{2} =  \mathrm{def}(\pi_1(K)) = 0.\] 
      So $k$ has to be $1$, i.e. $\pi_1(K)$ is a finite cyclic group without $p$-torsion.
      Then the corollary follows from Theorem~\ref{Thm:Mod-p-2dim} below.
    \end{proof}

  \begin{thm}[see~\cite{DyerSier73}] \label{Thm:Mod-p-2dim}
    Let $K$ be a connected $2$-dimensional finite CW-complex with
      fundamental group $\Z_m$. Then $X$ 
    has the homotopy type of the wedge sum $\mathbb{P}_m \vee S^2 \vee\cdots\vee S^2$
     of the pseudo-projective plane $\mathbb{P}_m$ and $rank$ $H_2(K)$-copies of 
     the $2$-sphere $S^2$.
  \end{thm}
  \nn

\section{$\F$-gap of a group}
    For any finitely presentable group $G$ and a field $\F$,
    Lemma~\ref{Lem:Deficiency} tells us that       
         $$ \mathrm{def}(G) \leq b_1(G;\F) - b_2(G;\F).$$
   Notice that both sides of the inequality are invariants of the group $G$. 
  Let
      $$ \mathrm{gap}(G;\F) := b_1(G;\F) - b_2(G;\F) - \mathrm{def}(G) \in \Z_{\geq 0}. $$ 
     We call $\mathrm{gap}(G;\F)$ the \emph{$\F$-gap} of $G$ which
      is also an invariant of $G$.

  \begin{exam} \label{Exam:Gap-Abelian}
     Let $A=\Z^{r} \oplus \Z_{d_1}\oplus \Z_{d_2}\oplus\cdots \oplus \Z_{d_k}$  
   where $d_1>1$ and $d_1 | d_2 | \cdots | d_k$. 
   First of all, we compute $\mathrm{gap}(A, \Q)$ where $\Q$ is the rational numbers.
   By Lemma~\ref{Lem:Deficiency-Abelian} and Equation~\eqref{Equ:H2-Abelian}, we obtain
   \[ \mathrm{gap}(A, \Q)= r -\binom{r}{2} - \left(r-\binom{r+k}{2} \right) 
    = \binom{r+k}{2} - \binom{r}{2}.
    \] 
    So $\mathrm{gap}(A, \Q)=0$ if and only if $k=0$ (i.e. $A$ is free abelian) 
    or $r=0$ and $k=1$ (i.e. $A$ is 
    a cyclic group).
    \n
   
   Next, we compute $\mathrm{gap}(A,\F_p)$ for any prime $p$.\n
   \begin{itemize}
    \item If $k=0$, i.e. $A\cong \Z^r$,  $\mathrm{gap}(A;\F_p) =0$.\n
    
    \item If $k\geq 1$, assume that $p | d_l $ but $p\nmid d_{l-1}$ 
     for some $1 \leq l \leq k$, then 
        $$b_1(A;\F_p)= r+k-l+1,$$ 
   $$ b_2(A;\F_p)= (k-l+1) + \binom{r}{2} + \sum^k_{j=l} (r+k-j).$$
    So by Lemma~\ref{Lem:Deficiency-Abelian}, we obtain
    \begin{align*}
       \mathrm{gap}(A;\F_p) &= \binom{r+k}{2} -  \binom{r}{2} - \sum^k_{j=l} (r+k-j) \\
        &=  \frac{1}{2} (2r+2k-l) (l-1).
    \end{align*}   
    Then $\mathrm{gap}(A;\F_p)=0$ if and only if $l=1$, i.e. 
    the torsion of $A$ is a $p$-group.\n
    
  \item If $k\geq 1$ and $p\nmid d_i$ for all $1\leq i \leq k$, we have 
  $\mathrm{gap}(A;\F_p) = \binom{r+k}{2} -  \binom{r}{2}$.
  In this case,  $\mathrm{gap}(A;\F_p) = 0$ if and only if $r=0$ and $k=1$, i.e. $A$ is 
  a finite cyclic group with no $p$-torsion.
    \end{itemize}
    
   By the above discussion, 
 both $\mathrm{gap}(A;\Q)$ and $\mathrm{gap}(A,\F_p)$ can take arbitrarily large values
  among finitely generated abelian groups.
 \end{exam}
   
   \n

   Next, let us interpret $\F$-gap of a group from some other viewpoints.
   
   \begin{prop} \label{Prop:Gap-G}
     For a finitely presentable group $G$ and a field $\F$,
          \begin{align*}
            \mathrm{gap}(G;\F) &= \mathrm{min}\{ b_2(K_{\mathcal{P}};\F) \, ;  
     \, \mathcal{P} \ \text{is a finite presentation of}\ G \} - b_2(G;\F) \\
     &= \mathrm{min}\{ b_2(K;\F) \, ;  
     \, K \ \text{is a connected finite $2$-complex with}\ \pi_1(K)\cong G \}\\
     & \ \ \ \  - b_2(G;\F).
     \end{align*}
    \end{prop}
    \begin{proof}
       It is sufficient to prove the first equality of the proposition.
      Note that for any finite presentation $\mathcal{P}$ of $G$,
      $ b_1(K_{\mathcal{P}};\F) = b_1(G;\F)$. 
      So we have
         \begin{align*} 
          b_2(K_{\mathcal{P}};\F) - b_2(G;\F) &= 
         b_2(K_{\mathcal{P}};\F) - b_1(K_{\mathcal{P}};\F) - (b_2(G;\F) -b_1(G;\F))\\
         &= b_1(G;\F)) - b_2(G;\F) - 
         (b_1(K_{\mathcal{P}};\F) - b_2(K_{\mathcal{P}};\F))
          \end{align*}
          Observe that $b_1(K_{\mathcal{P}};\F) - b_2(K_{\mathcal{P}};\F)$
          coincides with $\mathrm{def}(\mathcal{P})$. Then
                    \[ \mathrm{min}\{ - (b_1(K_{\mathcal{P}};\F) - 
                          b_2(K_{\mathcal{P}};\F)) \, ;\,
            \mathcal{P} \ \text{is a finite presentation of}\ G  \}  = - \mathrm{def}(G). \]
        This implies the first equality of the proposition.
    \end{proof}

    For any group $G$, the \emph{geometric dimension} $\mathrm{gd}(G)$ of $G$ is 
    the minimal dimension of the $K(G,1)$-complex. 
    The \emph{cohomological dimension} $\mathrm{cd}(G)$ of $G$ is the minimal length of a
    projective resolution of (the trivial $\Z[G]$-module) $\Z$. It is easy to see that
    $\mathrm{cd}(G) = \mathrm{cd}(K(G,1))$ and $\mathrm{gd}(G) \geq \mathrm{cd}(G)$.\n
     
       The following corollary is immediate from Proposition~\ref{Prop:Gap-G}.\n
    
   \begin{cor} \label{Cor:gd(G)}
   Let $G$ be a finitely presentable group
  with $\mathrm{gap}(G;\F)>0$ for some field $\F$. Then
   the geometric dimension of $G$ is at least $3$.
   \end{cor}
  \n
   \begin{prop}  \label{prop:Interp-Gap}
  For a finitely presentable group $G$, the following statements are equivalent.
  \begin{itemize}
    \item[(i)] $\mathrm{gap}(G;\F)=0$.\n
    
    \item[(ii)] There exists a $K(G, 1)$ complex $X$ so that
  $X$ has only finitely many cells in each dimension and 
   the cellular boundary map $\partial_3 : C_3(X;\F) \rightarrow C_2(X;\F)$ is trivial.
   \item[(iii)] There exists a finite presentation $\mathcal{P}$ of $G$ so that 
         the map $h: \pi_2(K_{\mathcal{P}}) \rightarrow H_2(K_{\mathcal{P}},\F)$ 
          induced from the Hurewicz map is the zero map.
         \end{itemize}
   \end{prop}
   \begin{proof}
  First, let us assume $\mathrm{gap}(G;\F)=0$.
  Then by Proposition~\ref{Prop:Gap-G},
    there exists a connected finite $2$-complex $K$ with
   $b_2(K;\F) = b_2(G;\F)$. Then by attaching finitely many $3$-cells, we can
   kill $\pi_2(K)$. Note that this process will not reduce $b_2(K;\F)$ because otherwise
   $b_2(G;\F)$ would be reduced as well, which is absurd. 
   Furthermore, by adding some higher dimensional cells (finitely many in each dimension),
   we will obtain a $K(G, 1)$ space that satisfies the requirements in (ii).\n
   Conversely, suppose $X$ is a $K(G, 1)$ space
    with only finitely many cells in each dimension and
  $\partial_3 : C_3(X;\F) \rightarrow C_2(X;\F)$ is trivial.  
  Then the $2$-skeleton $X^{(2)}$ of $X$ is a 
   connected finite $2$-complex with $\pi_1(X^{(2)}) \cong G$ and $b_2(X^{(2)};\F) = b_2(G;\F)$. 
  This implies that $\mathrm{gap}(G;\F)=0$ by Proposition~\ref{Prop:Gap-G}. So the equivalence of 
  (i) and (ii) is proved.
  In addition, $X^{(2)}$ is clearly homotopy equivalent to 
   the presentation complex of some finite presentation of $G$.
    Then the equivalence of (ii) and (iii) follows easily from the construction of $X$.
   \end{proof}
   
     \begin{rem}
      For a prime $p$, a group $G$ with $\mathrm{gap}(G;\F_p)=0$ is called 
      \emph{$p$-efficient} (see~\cite{Epstein60}).
       A finite presentation $\mathcal{P}$ of $G$ is called $p$-Cockcroft 
      (see~\cite{KilPride96}) if $\mathcal{P}$ satisfies the condition in 
      Proposition~\ref{prop:Interp-Gap} (iii).
     So the $\F$-gap can be thought of as a generalization of
     $p$-efficiency and the existence of $p$-Cockcroft presentation of a group $G$.
    \end{rem}

     The $\Q$-gap of a group is related to
      two other known concepts.
      A finitely presentable group $G$ is called \emph{efficient} if 
      there exists a connected finite $2$-complex $K$ with $\pi_1(K)\cong G$
      and $b_2(K;\Q) = d(H_2(G))$ where $d( - )$ is the minimal number of 
      generators of a group. Indeed, we always have
      $b_2(K;\Q) \geq d(H_2(G))$ since $H_2(K)$ is free abelian and $H_2(G)$
      is a quotient of $H_2(K)$ (see~\eqref{Equ:Hopf-1}). In addition,
      a connected $2$-complex is said to have the \emph{Cockcroft property} 
      if the Hurewicz map $ h: \pi_2(K) \rightarrow H_2(K) $ is the zero map.
      The Cockcroft property was first 
      studied by Cockcroft~\cite{Cock51} in connection with the \emph{Whitehead conjecture} 
      (which states that a subcomplex of an aspherical 2-complex is itself aspherical).
      \n      
   \begin{prop} \label{Prop:Efficient}
    For any finitely presentable group $G$, the following statements are equivalent.
    \begin{itemize}
    \item[(i)] $\mathrm{gap}(G;\Q)=0$.\n
    \item[(ii)] $G$ is efficient and $H_2(G)$ has no torsion.\n
    \item[(iii)] $G$ has a finite presentation $\mathcal{P}$ whose presentation complex 
     $K_{\mathcal{P}}$ has the Cockcroft property.
    \end{itemize}
   \end{prop}
    \begin{proof}
     These equivalences follow easily from the Hopf's 
     exact sequence~\eqref{Equ:Hopf-1} and Proposition~\ref{Prop:Gap-G}.
    \end{proof}
    
   By the above proposition,
    the set of non-efficient finitely presentable groups is included in 
     the set of groups with nonzero $\Q$-gaps. Moreover, this inclusion is strict.
    For example,   
    any finitely generated abelian group $ A$ is efficient. But our calculation in 
    Example~\ref{Exam:Gap-Abelian} shows that
    $\mathrm{gap}(A;\Q) > 0$ if $A$ is not free abelian or cyclic (where $H_2(A)$ 
    has torsion).\n

   The notion of $\Q$-gap is also related to the Eilenberg-Ganea conjecture.
    The \emph{Eilenberg-Ganea conjecture} claims that a group 
 of cohomological dimension $2$ is of geometric dimension $2$. 
  Let $G$ be a finitely presentable group with
  $\mathrm{cd}(G)=2$. 
  
 \begin{itemize}
   \item  If $\mathrm{gap}(G;\Q) =0$, then $G$ is efficient by Proposition~\ref{Prop:Efficient}.  
   It is shown in~\cite{Hillman89} that when $\mathrm{cd}(G)=2$, $G$ is efficient if and only if 
   there exists a finite $2$-dimensional $K(G,1)$-complex. So we obtain $\mathrm{gd}(G)=2$.\n
   
   \item If $\mathrm{gap}(G;\Q) > 0$, $\mathrm{gd}(G)\geq 3$ 
   by Corollary~\ref{Cor:gd(G)}. Then $G$ is a counterexample of 
   the Eilenberg-Ganea conjecture.   
 \end{itemize}  
   
  So for a finitely presentable group $G$, the Eilenberg-Ganea conjecture is equivalent to say
  that $\mathrm{cd}(G) = 2$ implies $\mathrm{gap}(G;\Q) =0$.
   \n
    
    \begin{exam}
     If a finite $2$-complex $K$ is 
    an $\F$-homology rose or 
     $\F$-acyclic space, then $\mathrm{gap}(K;\F)=0$.
    In particular, for any knot $N$ in $S^3$, 
       $\mathrm{gap}(\pi_1(S^3\backslash N);\F)=0$ for any field $\F$ since
        $S^3\backslash N$ is homotopy equivalent to a finite $2$-complex. 
   \end{exam}

   \begin{exam} \label{Exam:3-manifolds}
    Let $M^3$ be a closed connected 
     $3$-manifold. If $M^3$ is orientable,
     there exists a finite cell decomposition
     of $M^3$ with only a single $3$-cell. Let $K$ be the $2$-skeleton of this cell decomposition.
    Since $M^3$ is orientable, the boundary of the $3$-cell is $0$. This implies that
      $H_2(K;\F)=H_2(M^3;\F)$ for any field $\F$. \n
     \begin{itemize}
    \item If $M^3$ is aspherical, $b_2(\pi_1(M^3);\F) =b_2(M^3;\F)$.
     Then since
          $$\qquad \quad \ \ \ \mathrm{gap}(\pi_1(M^3);\F) \leq b_2(K;\F) - b_2(\pi_1(M^3);\F) =
          b_2(K;\F) - b_2(M^3;\F) = 0,$$
        we obtain $\mathrm{gap}(\pi_1(M^3);\F)=0$.\n  
    
    \item If $M^3$ is a spherical manifold (i.e. the universal cover of $M^3$ is $S^3$),
            we also have $\mathrm{gap}(\pi_1(M^3);\F)=0$. 
            This is because $\pi_2(M^3)\cong \pi_2(S^3)=0$, hence $M^3$ can serve as
            the $3$-skeleton of a $K(\pi_1(M^3),1)$ space.
            \n
             
     \end{itemize}   
     If $M^3$ is non-orientable and aspherical, we can show that 
     $\mathrm{gap}(\pi_1(M^3);\F_2)=0$ by the same argument. 
     For a general field $\F$, we can only show
      $\mathrm{gap}(\pi_1(M^3);\F) \leq 1$.  \n
   \end{exam}

   Next, we discuss the relation between the $\mathbb Q$-gap and an arbitrary $\mathbb F$-gap 
   of a group.
   For a finitely presentable group $G$, let
    $$H_{i}(G;\mathbb Z) \cong \mathbb Z^{b_{i}} \oplus T_{i},\ \, i=1,2,$$
     where  $b_i = b_{i}(G;\mathbb Q)$ and 
   $T_{i}$ stands for the torsion subgroup of $G$. So we have 
 \begin{equation}  \label{Equ:H_1-Relation}
     H_{1}(G; \mathbb F) \cong H_{1}(G) \otimes \mathbb F \cong \mathbb F^{b_{1}} 
     \oplus (T_{1} \otimes \mathbb F),
 \end{equation}    
   \begin{equation} \label{Equ:H_2-Relation}
      H_{2}(G;\mathbb F) \cong H_{2}(G) \otimes \mathbb F 
      \oplus \mathrm{Tor}_{\mathbb Z}(H_{1}(G); \mathbb F) \cong \mathbb F^{b_{2}} 
      \oplus (T_{2} \otimes \mathbb F) \oplus \mathrm{Tor}_{\mathbb Z}(T_{1};\mathbb F).
   \end{equation}
Note that any finite abelian group $T$ can be defined by an exact sequence
 $$ 0 \to \mathbb Z^{n} \to \mathbb Z^{n} \to T \to 0. $$ 
 Then we have an exact sequence
$0 \to \mathrm{Tor}_{\mathbb Z}(T, \mathbb F) \to \mathbb F^{n} 
 \to \mathbb F ^{n} \to T \otimes \mathbb F \to 0$.
This implies that $\dim_{\mathbb F}\mathrm{Tor}_{\mathbb Z}(T, \mathbb F) =\dim_{\mathbb F} T \otimes \mathbb F$. So we obtain
 \begin{equation}\label{Equ:Relation_gap}
    \mathrm{gap}(G;\mathbb F)=b_{1}(G;\mathbb F) - b_{2}(G;\mathbb F)- \mathrm{def}(G)
 =\mathrm{gap}(G;\mathbb Q) - \dim_{\mathbb F}(T_{2} \otimes \mathbb F).
 \end{equation}
  This relation implies the following.\n
  
  \begin{prop} \label{Prop:Relation-gap}
 For a finitely presentable group $G$, $\mathrm{gap}(G;\mathbb Q)=0$ implies that
   $\mathrm{gap}(G;\mathbb F)=0$ for any field $\mathbb F$;
   and $\mathrm{gap}(G;\mathbb Q) >0$ implies that $\mathrm{gap}(G;\mathbb F_{p}) > 0$ 
   for almost all prime $p$. 
  \end{prop}

     Finally, let us discuss the property of $\F$-gap under
  the free product and direct product of 
  groups. Let $G_1$ and $G_2$ be two finitely presentable groups.\n
  
  \begin{itemize}
    \item If $\mathrm{gap}(G_1;\F) = \mathrm{gap}(G_2;\F)=0$, we must have 
    $\mathrm{gap}(G_1*G_2;\F) = 0$ and
    $\mathrm{def}(G_1*G_2) = \mathrm{def}(G_1)+\mathrm{def}(G_2)$. The proof is 
    parallel to the proof of~\cite[Lemma 1.6]{Epstein60}.\n

    \item That
  $\mathrm{gap}(G_1;\F)= \mathrm{gap}(G_2;\F)=0$ does not
  imply $ \mathrm{gap}(G_1\times G_2;\F)=0$. For example, for $d_1,d_2>1$ with 
   $p\nmid d_1$ and $p\nmid d_2$, $\mathrm{gap}(\Z_{d_1};\F_p)=0$ and 
   $\mathrm{gap}(\Z_{d_1d_2};\F_p)=0$. But by the calculation in 
   Example~\ref{Exam:Gap-Abelian}, we have
   $$ \mathrm{gap}((\Z_{d_1})^2;\F_p)=1,
    \ \, \mathrm{gap}(\Z_{d_1} \times \Z_{d_1d_2};\F_p)=1.$$
     Conversely, we may ask the following question.\nn
  \end{itemize}
  
 \textbf{Question:}
    If $\mathrm{gap}(G_1\times G_2;\F)=0$, 
  is it necessary that both $\mathrm{gap}(G_1;\F)=0$ and $\mathrm{gap}(G_2;\F)=0$.
  \n
  The answer to this question is yes when $G_1$ and $G_2$ are both finitely generated abelian groups (see Example~\ref{Exam:Gap-Abelian}). But the general answer is not clear to us.\\

   \section{Digression to the $\mathrm{D(2)}$-problem}
    By Proposition~\ref{Prop:Fund-Group} and Proposition~\ref{Prop:mod-p-Acyclic},
   any finitely presentable group $\Gamma$ with $b_2(\Gamma;\F)=0$ can be realized as
    the fundamental group of a finite $3$-complex $B$ which is an 
    $\F$-homology rose or an $\F$-acyclic space.
     But generally speaking, we can not
     require $B$ to be $2$-dimensional.
      Indeed, Proposition~\ref{Prop:Gap-G} implies the following.\nn
      
   \begin{prop}\label{Prop:2-dim-gap}
    For any field $\F$, 
    a finitely presentable 
    group $\Gamma$ can be realized as the fundamental group of a connected finite $2$-complex $K$
    with $b_2(K;\F)=0$ if and only if $b_2(\Gamma;\F)=0$ and
    $\mathrm{gap}(\Gamma;\F) = 0$. \n
   \end{prop}   

    So if $B$ is a $3$-complex with
       $b_2(B;\F)=0$ and $\mathrm{gap}(\pi_1(B);\F) > 0$, then
       $B$ can not be homotopy equivalent to a $2$-complex. More generally,
       we have:
     \n
     
     \begin{lem} \label{Lem:Principal}
       If $B$ is a connected finite $3$-complex with
       $b_2(B;\F)=b_2(\pi_1(B);\F)$ and $\mathrm{gap}(\pi_1(B);\F) > 0$, then
       $B$ is not homotopy equivalent to any finite $2$-complex.
     \end{lem}  
       
       This simple fact motivates us to study
      Wall's $\mathrm{D(2)}$-problem (see~\cite{Wall65}) among
      $\F$-homology roses and $\F$-acyclic spaces.
             \nn
             
         \noindent \textbf{$\mathbf{D(2)}$-problem:} 
   Let $B$ be a finite $3$-complex with $\mathrm{cd}(B)=2$, i.e.
    $H^3(B; \mathcal{M}) = 0$ for all $\pi_1(B)$-module $\mathcal{M}$; is it true that $B$ is 
 homotopy equivalent to a finite $2$-complex?
      \n
      We say that the
      \emph{$\mathrm{D(2)}$-property} holds for a group $\Gamma$ if any connected 
      finite $3$-complex $B$ with $\mathrm{cd}(B)=2$ and $\pi_1(B)\cong \Gamma$
       is homotopy equivalent to a finite $2$-complex.
       The $\mathrm{D(2)}$-problem is equivalent to asking whether the 
       $\mathrm{D(2)}$-property holds for all finitely 
       presentable groups. So far the $\mathrm{D(2)}$-property is known to hold for\n
       \begin{itemize}
        \item free groups (Johnson~\cite{John03}),\n
        \item finite abelian groups (Browning, Latiolais~\cite{Browning78, Lat86}),\n
        \item $\Z\times \Z_n$ and $\Z^2$ (Edwards~\cite{Edward06}), \n
        \item dihedral groups of order $8$ 
        (Mannan~\cite{Man07}) and of order $4n+2$ (Johnson~\cite{John02}).
      \end{itemize} 
        \n
         On the other hand, no counterexamples of the $\mathrm{D(2)}$-problem are known so far,
       although people
      have tried many methods and proposed some candidates for such counterexamples
         (see~\cite{BriTwe07, Edward06, Harland98, John03, John04, Man09, RudWal08}). 
         Our discussion on $\F$-gaps 
         motivates us to search candidates for the counterexamples of 
     the $\mathrm{D(2)}$-problem among $\F$-homology roses and $\F$-acyclic spaces via the following strategy.\n   
     
      \textbf{Strategy:}\n
      \begin{itemize}
      \item Take a finitely presentable group $\Gamma$ with
       $H_2(\Gamma;\F)=0$ and $\mathrm{gap}(\Gamma;\F) > 0$.\n
      
      \item 
      By Proposition~\ref{Prop:Fund-Group} and Proposition~\ref{Prop:mod-p-Acyclic}, 
      there exists a connected finite $3$-complex $B$ with 
      $\pi_1(B) \cong \Gamma$ and $b_2(B;\F)=0$.
      If we can choose $B$ with $\mathrm{cd}(B) = 2$, then
      $B$ is a counterexample of the $\mathrm{D(2)}$-problem.
    \end{itemize}\n

   To actually put the above strategy into practice,
    we need
    to first construct enough examples of finite $3$-complexes
    with cohomological dimension two.
    A general way to do so is to use 
    Quillen's plus construction (see~\cite{Man09}).
         \n
     
      Let $K$ be a connected finite $2$-complex.
      Assume that the fundamental group $\pi_1(K)$ of $K$ has a 
    finitely closed perfect normal subgroup $N$. Here $N$ is \emph{finitely closed} means
    that it is the normal closure of a finitely generated subgroup of $\pi_1(K)$. And 
     $N$ is \emph{perfect} means $H_1(N;\Z)=0$ (i.e. $N =[N,N]$).   
    Then \emph{Quillen's plus construction} on $K$ gives us
     a finite $3$-complex $K^+$ with
         \begin{equation}
            H_*(K^+;\Z)\cong H_*(K;\Z), \ \pi_1(K^+) \cong \pi_1(K)\slash N,\ 
            \mathrm{cd}(K^+)=\mathrm{cd}(K)=2.
         \end{equation}        
       \n

   It is shown in~\cite[Theorem 3.4]{Man09} that any connected finite $3$-complex with 
    cohomological dimension $2$ can be obtained from a finite $2$-complex using
    the plus construction. So in particular, if a finite $3$-complex $B$ with
     $b_2(B;\F)=0$ has cohomological dimension $2$, then 
     there exists a finite $2$-complex $K$ so that $B$ is homeomorphic to 
     $K^+$ with respect to some finitely closed 
     perfect normal subgroup of $\pi_1(K)$. This implies that
     $b_2(K;\F)=0$, and so $b_2(\pi_1(K);\F)=0$ and
     $\mathrm{gap}(\pi_1(K);\F) = 0$ (by Proposition~\ref{Prop:2-dim-gap}).\nn
     
     So if we assume that a finite $3$-complex 
     $B$ is a counterexample of the $\mathrm{D(2)}$-problem and $B$
      is an $\F$-homology rose or an $\F$-acyclic space, we necessarily have
     an exact sequence of the following form. 
      \begin{equation}\label{Equ:D(2)-CoutExam}
        1 \longrightarrow N \longrightarrow  G \longrightarrow  \Gamma \longrightarrow 1 
        \ \ \text{where}
     \end{equation}
    \begin{itemize}
      \item $N$ is a nontrivial finitely closed perfect normal subgroup of $G$. \n
      \item $G$ is a finitely presentable group with $b_2(G;\F)=0$, $\mathrm{gap}(G;\F) = 0$.\n
      \item $b_2(\Gamma;\F)=0$, $\mathrm{gap}(\Gamma;\F) > 0$, $H_1(\Gamma;\Z)=H_1(G;\Z)$.
    \end{itemize}  
      These conditions further imply that 
    \begin{itemize}
     \item $\mathrm{def}(G) = b_1(G;\F) = b_1(\Gamma;\F)   \geq 0$ and 
    $\mathrm{def}(G) > \mathrm{def}(\Gamma)$. \n
    \end{itemize}

    Conversely, Proposition~\ref{Prop:Fund-Group} implies that
     any sequence as in~\eqref{Equ:D(2)-CoutExam} gives us 
     a counterexample $B$ of the $\mathrm{D(2)}$-problem where 
     $B$ is an $\F$-homology rose or an $\F$-acyclic space with $\pi_1(B)\cong \Gamma$.\n   
          
  Note that in~\eqref{Equ:D(2)-CoutExam},
   $b_{2}(G;\mathbb F)=0$ implies $b_{2}(G;\mathbb Q)=0$ and 
  $\mathrm{tors}(H_{2}(G)) \otimes \mathbb F=0$ (see~\eqref{Equ:H_2-Relation}).
  Then since $\mathrm{gap}(G;\mathbb F)=0$, we get $\mathrm{gap}(G;\mathbb Q)=0$ (see~\eqref{Equ:Relation_gap}). 
  Similarly, $b_2(\Gamma;\F)=0$ implies
   $b_{2}(\Gamma; \mathbb Q)=0$ and $\mathrm{gap}(\Gamma;\mathbb Q)
   =\mathrm{gap}(\Gamma;\mathbb F) > 0$. 
  So if we want to search counterexamples of the $\mathrm{D(2)}$-problem via
   the sequence~\eqref{Equ:D(2)-CoutExam},
     we only need to consider $\mathbb Q$-coefficients.\n

      We can think of $\Gamma$ in the exact sequence~\eqref{Equ:D(2)-CoutExam} 
       as the deck transformation group of a regular 
        covering space $p: X \rightarrow K$ where $X$ and $K$ are connected
        finite $2$-complexes with $\pi_1(X)\cong N$ and $\pi_1(K)\cong G$. 
        From this viewpoint,
        we can try the following ways
         to construct possible counterexamples of the $\mathrm{D(2)}$-problem.
        \nn  
        
     \begin{itemize}
       \item[(M1)] Let $X$ be a connected $2$-complex
          with $H_1(X;\Z)=0$. Suppose $\Gamma$ is a group with $\mathrm{gap}(\Gamma;\Q) > 0$.
          If $\Gamma$ can act freely on $X$ so that  
          the orbit space $K=X\slash \Gamma$
          is a connected finite $2$-complex with $b_2(K;\Q) =0$, then
          the plus-construction $K^+$ with
          respect to $\pi_1(X)$ is a counterexample of the $\mathrm{D(2)}$-problem. \n
          
          \item[(M2)] Let $K$ be a connected 
          finite $2$-complex with $b_2(K;\Q) =0$ and let
           $\Gamma$ be a
          group with $\mathrm{gap}(\Gamma;\Q) > 0$. If we can construct a regular 
          $\Gamma$-covering
          space $X$ over $K$ with $H_1(X;\Z)=0$, the plus construction $K^+$ 
          with respect to $\pi_1(X)$ is 
          a counterexample of the $\mathrm{D(2)}$-problem.\n
        \end{itemize}                   
                     
       Moreover, the above way of searching candidates for the counterexamples of the
        $\mathrm{D(2)}$-problem can be generalized as follows.\n
        
         \begin{itemize}
          \item[(M3)] Let $G$ be a finitely presented group.
            If there exists a finitely closed perfect normal subgroup $N$ 
            of $G$ so that the quotient group
            $\Gamma=G\slash N$ satisfies $\mathrm{gap}(\Gamma;\Q) > 0$ and 
            $b_2(\Gamma;\Q)= b_2(K_{\mathcal{P}};\Q)$, where $\mathcal{P}$
            is a finite presentation of $G$,
           then the plus construction $K^+_{\mathcal{P}}$ with respect to $N$ is a
          finite
           $3$-complex with $b_2(K^+_{\mathcal{P}};\Q)=b_2(\Gamma;\Q)$. So $K^+_{\mathcal{P}}$
          is a counterexample of the $\mathrm{D(2)}$-problem. \n
          \end{itemize}
        
        \begin{rem}
        The above way of searching counterexamples of the $\mathrm{D(2)}$-problem
        is essentially equivalent to the one proposed in~\cite[Theorem 4.5]{Man09}.\n
        \end{rem} 
              
  The following discussion tells us that
         if we want to search counterexamples of the $\mathrm{D(2)}$-problem among
        $\F$-homology roses, we need to assume, a priori, that the fundamental groups
        of the $\F$-homology 
        roses are infinite, non-abelian and have trivial multiplicators.\n
        
          \begin{prop} \label{prop:Homo-Except-Cases}
       Let $B$ be a finite $3$-complex.
       \begin{itemize}
        \item[(i)] If $B$ 
       is an $\mathbb F$-homology rose and $\pi_1(B)$ is finite, $\mathrm{cd}(B)=3$. \n
       
       \item[(ii)] If $B$ 
       is an $\mathbb F$-homology rose or an $\F$-acyclic space with $H_2(\pi_1(B))\neq 0$, then
               $\mathrm{cd}(B)=3$.
       \end{itemize}              
          \end{prop}
          \begin{proof}
     (i) Consider the universal cover $\widetilde B \to B$.
      Since $\chi(B)=1-b_{1}(B;\mathbb F) \leq 0$, 
      $\chi(\widetilde B)=|\pi_1(B)| \cdot \chi(B) \leq 0$. But since $\widetilde{B}$ is 
      simply connected,
      $$ \chi(\widetilde B)=1+b_{2}(\widetilde B;\Q)-b_{3}(\widetilde B;\Q).$$
      So $H_{3}(\widetilde B) \neq 0$. Then since
       $H^{3}(B; \mathbb Z[\pi_{1}(B)]) = H^{3}(\widetilde B) \neq 0$, 
       $\mathrm{cd}(B)=3$.\n
       
      (ii)  
       Since $H_2(B;\F)=0$, we have $H_2(B;\Q)=0$. So $H_2(B)$ is finite abelian.
         Then since $H_2(\pi_1(B)) \neq 0$, 
      the sequence~\eqref{Equ:Hopf-1} implies that $H_2(B)$ is non-trivial.
       Hence $H_2(B)$ is a non-trivial finite abelian group. This implies    
           $\mathrm{cd}(B)=3$. 
   \end{proof}
   
  \begin{cor}
   Suppose a finite $3$-complex $B$ with $\mathrm{cd}(B)=2$
   is an $\F$-homology rose or an $\F$-acyclic space.
    Then the following statements are equivalent.
      \begin{itemize}
         \item[(i)] $\pi_1(B)$ is abelian.\n
         \item[(ii)] $\pi_1(B)$ is a cyclic group.\n
         \item[(iii)] $B$ is homotopy equivalent to $S^1$ or $\mathbb{P}_m$, $m\geq 2$.
       \end{itemize}
    \end{cor}     
      \begin{proof}
      Since $\mathrm{cd}(B)=2$, we have 
      $H_2(\pi_1(B))=0$ by Proposition~\ref{prop:Homo-Except-Cases}.
       For a finitely generated abelian group $A$, $H_2(A) = 0$ if and only if 
       $A$ is cyclic (see~\eqref{Equ:H2-Abelian}). 
       So we can derive (ii) from (i) .
       Moreover, since the $\mathrm{D(2)}$-property holds for all cyclic groups, 
      (ii) implies that $B$ is homotopy equivalent to a finite $2$-complex, which further
      implies (iii) according to Corollary~\ref{Cor:Homology-Rose-S1} 
       and Corollary~\ref{Cor:Mod-Acyclic-Fund}.                 
      \end{proof}
      
      \nn
      
      \nn

  \section{Search among $3$-manifold groups}
  
           If we assume that the group $G$ in the exact sequence~\eqref{Equ:D(2)-CoutExam} 
           is the fundamental group of a closed connected $3$-manifold $M^3$,
           then~\cite[Theorem 4.2]{BerHill03} 
           implies
           
            \begin{itemize}
             \item $M^3$ is a rational homology $3$-sphere, and so $b_1(G;\Q) = b_2(G;\Q)=0$.\n
             \item $N$ is the fundamental group of an integral homology 
             $3$-sphere $\hat{M}^3$.\n
             \item $\Gamma$ is a finite group of periodic cohomology with period $1$, $2$ or $4$,
             and $b_1(\Gamma;\Q)=b_2(\Gamma;\Q)=0$ (since $\Gamma$ is the fundamental group 
             of the plus-construction $(M^3)^+$).
           \end{itemize}
           \n
           Recall that a finite group $\Gamma$ is \emph{periodic} and of period $m >0$ if and only
           if $H^i(\Gamma;\Z) \cong H^{i+m}(\Gamma;\Z)$ for all $i\geq 1$ where $\Gamma$ acts
           on $\Z$ trivially. \nn
                      
           $M^3$ is the orbit space of a free orientation-preserving action of $\Gamma$ on 
           $\hat{M}^3$.        
          Note that the existence of Heegaard splittings for closed $3$-manifolds implies that
            $\mathrm{def}(G) = \mathrm{def}(\pi_1(M^3))\geq 0$. Then 
            since $b_1(G;\Q) = 0$, the inequality~\eqref{Def:Deficiency-G} implies 
            \[ \mathrm{def}(G)=0,\ H_2(G;\Z) =0.  \]
            
             If $\hat{M}^3$ is the $3$-sphere $S^3$, 
            $\Gamma$ is the fundamental group of a spherical $3$-manifold
            and hence $\mathrm{gap}(\Gamma;\Q)=0$ (see Example~\ref{Exam:3-manifolds}). 
            So if we want to have $\mathrm{gap}(\Gamma;\Q) > 0$,  
            $\hat{M}^3$ has to be an integral homology sphere other than $S^3$.
               \n

           \begin{prop}\label{Prop:Hom-sphere}
             If a finite group $\Gamma$ with $\mathrm{def}(\Gamma)<0$ 
             can act freely on an integral homology $3$-sphere $\Sigma^3 \ncong S^3$, 
             then it gives a counterexample of 
                the $\mathrm{D(2)}$-problem.
           \end{prop}
           \begin{proof}
              Since $\Gamma$ is finite, the orbit space 
              $\Sigma^3\slash \Gamma$ is a rational homology $3$-sphere.
             Since $b_1(\Gamma;\Q)=b_2(\Gamma;\Q) =0$, $\mathrm{def}(\Gamma)<0$ implies that
             $\mathrm{gap}(\Gamma; \Q) > 0$.
               Let $K$ be the $2$-skeleton of $\Sigma^3\slash \Gamma$. It is clear that
              $K$ is a $2$-dimensional $\Q$-acyclic space. 
              Then by our discussion,
              the plus construction $K^+$ with respect to the nontrivial group
              $\pi_1(\Sigma^3) <  \pi_1(\Sigma^3\slash \Gamma) = \pi_1(K)$ 
              is a counterexample of the $\mathrm{D(2)}$-problem.
           \end{proof}
           
            By Smith's theory, if a finite group 
            $\Gamma$ can act freely on an integral homology $3$-sphere, it is necessary that
            any elementary abelian $p$-subgroup of $\Gamma$ is cyclic for all prime $p$.
            This is equivalent to say that
             any abelian subgroup of $\Gamma$ is cyclic (see~\cite[p.157]{Brown82}). 
             Such finite groups are all periodic and have been
             classified by Suzuki and Zassenhaus 
            (see~\cite[p.154]{AdemMilg94} or~\cite[Chapter 6.3]{Wolf74}).
            \n
            
            But it is not true that any
            finite group whose abelian subgroups are all cyclic can act freely  
          on an integral homology sphere. An additional restriction is 
          the Milnor condition (see~\cite{Milnor57}) 
          which says that if a group $\Gamma$ can act freely on an integral homology sphere, 
          $\Gamma$ has at most one element of order two. 
          This excludes, for example, dihedral groups.
          A list of the possible finite groups which may act freely on an integral homology 
          $3$-sphere is given in Milnor's paper~\cite{Milnor57}. 
          But the complete classification of
          such groups remains open (see~\cite{DavMilg85} and~\cite{Zimm11} 
          for more information). 
          The difficulty lies in the understanding of a family of groups
           defined below. For relatively coprime positive integers $8n$, $k$ and $l$,
           let $Q(8n, k, l)$ denote the group with presentation
            \[ \langle x, y, z \,|\, x^2=(xy)^2 = y^{2n},  
            z^{kl} = 1, xzx^{-1} = z^r, yzy^{-1} = z^{-1}  \rangle \]
            where $r\equiv -1 \mod k$ and $r\equiv +1 \mod l$. 
            The product of $Q(8n, k, l)$ 
            with a cyclic group of coprime order are the only type of groups in 
            Milnor's list~\cite{Milnor57}
            which may possibly act freely on some integral homology $3$-sphere other than $S^3$.
            In addition, a result in~\cite{Lee73} further limits 
            the possibility to: $n$ is odd and $n > k > l \geq 1$.  \n
            
            Therefore, if we assume 
            the group $G$ in~\eqref{Equ:D(2)-CoutExam} to be the fundamental group
            of a closed $3$-manifold,
            the groups $\Z_m\times Q(8n, k, l)$ with $n$ odd and $n > k > l \geq 1$
           are the only candidates for $\Gamma$ for us to 
            construct counterexamples
            of the $\mathrm{D(2)}$-problem in this approach.
            \\

       \section{Search among finite groups with trivial multiplicator and negative deficiency}
       
  Let $\Gamma$ be a finite group with $H_2(\Gamma)=0$ and 
  $\mathrm{def}(\Gamma)<0$. Then 
  $$b_1(\Gamma;\Q) = b_2(\Gamma;\Q)=0, \ \ \mathrm{gap}(\Gamma;\Q) > 0.$$
  Let $K$ be a finite $2$-complex
   with $\pi_1(K)\cong \Gamma$. So by Hopf's
    exact sequence       
    \begin{equation*} 
        \pi_2(K) \longrightarrow H_2(K) \longrightarrow H_2(\Gamma) \longrightarrow 0, 
    \end{equation*}    
    we can attach some $3$-cells to $K$ to obtain a finite $3$-complex $B$ so that $B$
    is a
   $\Q$-acyclic space with $\pi_1(B)\cong \Gamma$ (see
    the proof of Proposition~\ref{Prop:Fund-Group}). Actually, we can require
     $H_2(B)=H_3(B)=0$.\nn
   
      In addition, Proposition~\ref{Prop:2-dim-gap} (or Proposition~\ref{prop:Deficiency-1}) 
  implies that $B$ can not be homotopy equivalent to a $2$-complex.
  So $B$ would be a counterexample of the $\mathrm{D(2)}$-problem if
   $\mathrm{cd}(B)=2$.\n

  Examples of finite groups with trivial multiplicator and negative deficiency
   can be found in Swan~\cite{Swan65} and Kov\'acs~\cite{Kov95}. \n
  
  The following examples are taken from~\cite{Swan65}.  
    Let $\Z_3=\langle x \rangle $ acts on $(\Z_7)^k$ by 
   $ xyx^{-1} = y^2 $ for all $y \in (\Z_7)^k$.
 Let $G_k$ be the corresponding semi-direct product of $\Z_3$ and $(\Z_7)^k$.
 It is shown in~\cite[p.197]{Swan65} that
 \begin{itemize}
  
  \item $H_2(G_k) =0$ for any $k$. \n
  
  \item $\mathrm{def}(G_k) <0$ for any $k\geq 3$ and
   $\mathrm{def}(G_k) \rightarrow -\infty$ as $k\rightarrow \infty$.\n
   
   \item $G_k$ is not efficient for any $k\geq 3$. \n
  \end{itemize}  
      
 So $ \mathrm{gap}(G_k;\Q) > 0 $ for all $k\geq 3$.
 By the preceding discussion, there exists a finite $3$-complex $B_k$ 
     with $\pi_1(B_k)\cong G_k$ and $B_k$ is a $\Q$-acyclic space. 
    When $k\geq 3$, $B_k$ 
  is not homotopy equivalent to a $2$-complex.\n
     
   Let $\widetilde{B}_k$ be the universal covering of $B_k$ whose cellular chain complex
    is: 
    \[ \cdots \ 0\longrightarrow C_3(\widetilde{B}_k) \overset{\tilde{\partial}_3}{\longrightarrow}  
          C_2(\widetilde{B}_k) \overset{\tilde{\partial}_2}{\longrightarrow}
           C_1(\widetilde{B}_k) \overset{\tilde{\partial}_1}{\longrightarrow}
            C_0(\widetilde{B}_k) \longrightarrow 0. \]
        To prove $\mathrm{cd}(B_k)= 2$,  it is sufficient to show that
    $\tilde{\partial}_3$ splits, i.e. 
    there exists a $\Z[\pi_1(B_k)]$-module homomorphism $\beta: C_2(\widetilde{B}_k) \rightarrow 
    C_3(\widetilde{B}_k)$ so that $\beta\circ 
    \tilde{\partial}_3 = \mathrm{id}_{C_3(\widetilde{B}_k)}$. 
   \n
   
  \noindent \textbf{Claim:}   $\mathrm{cd}(B_k)= 2$ if and only if $H^3(\widetilde{B}_k)=0$.\nn
    
   Since $\widetilde{B}_k$ is compact, we have
      $$ H^3(B_k,\Z[G_k]) \cong  H^3_c(\widetilde{B}_k) = H^3(\widetilde{B}_k)\ \,
      \text{(see~\cite[Proposition 3H.5]{Hat-book})}.$$
      So if $\mathrm{cd}(B_k)= 2$, we must have $H^3(\widetilde{B}_k)=0$. Conversely, assume
         $H^3(\widetilde{B}_k)=0$.
      We can think of $\mathrm{id}_{C_3(\widetilde{B}_k)}$ as an element of
      $H^3(B_k,C_3(\widetilde{B}_k))$. But  $H^3(B_k,C_3(\widetilde{B}_k))$ is zero
      because $C_3(\widetilde{B}_k)$ is a free $\Z[G_k]$-module and 
      $H^3(B_k,\Z[G_k])=0$. Therefore, there exists a $\Z[G_k]$-module homomorphism
            $\beta: C_2(\widetilde{B}_k) \rightarrow 
    C_3(\widetilde{B}_k)$ so that 
    $\beta\circ \tilde{\partial}_3 = \mathrm{id}_{C_3(\widetilde{B}_k)}$.
     Then $\partial_3$ is a split injection.
    The claim is proved.\n
    
     So to compute $\mathrm{cd}(B_k)$, it amounts to determine
    whether $H^3(\widetilde{B}_k,\Z)$ is zero. But the calculation is not
    clear to us.\\

     \section*{Acknowledgements}
   The authors want to thank Gongxiang Liu and Shengkui Ye for some helpful discussions.
   \\

\end{document}